\documentclass[11pt]{article}

\usepackage{amsfonts}
\usepackage{amsmath}

\usepackage[left=2cm,top=2cm,right=3cm]{geometry}
\newcommand{\beq}{\begin{equation}}
\newcommand{\eeq}{\end{equation}}
\newcommand{\bea}{\begin{eqnarray}}
\newcommand{\eea}{\end{eqnarray}}
\newcommand{\beas}{\begin{eqnarray*}}
\newcommand{\eeas}{\end{eqnarray*}}

\newtheorem{theorem}{Theorem}[section]

\newtheorem{definition}[theorem]{Definition}
\newtheorem{proposition}[theorem]{Proposition}

\newtheorem{lemma}[theorem]{Lemma}
\newtheorem{remark}[theorem]{Remark}
\newtheorem{example}[theorem]{Example}
\newtheorem{examples}[theorem]{Examples}
\newtheorem{foo}[theorem]{Remarks}

\newenvironment{proof}{\addvspace{\medskipamount}\par\noindent{\it
Proof}.}
{\unskip\nobreak\hfill$\Box$\par\addvspace{\medskipamount}}

\newcommand{\bG}{\mathbb G}

\title{Taylor expansion for the solution of a stochastic differential equation driven by fractional Brownian motions}
\author{Fabrice Baudoin\footnote{First author supported in part by
NSF Grant DMS 0907326}, Xuejing Zhang}
\date{Department of Mathematics, Purdue University \\
West Lafayette, IN, USA}

\begin{document}

\maketitle

\begin{abstract}
We study the Taylor expansion for the solution of a differential equation driven by a multi-dimensional H\"{o}lder path
with exponent $\beta>1/2$. We derive a convergence criterion that enables us to write the solution as an infinite sum of iterated integrals on a nonempty interval. We apply our deterministic results to  stochastic differential equations driven by fractional Brownian motions with Hurst parameter $H>1/2$.
We also prove that by using $L^2$ estimates of iterated integrals, the  criterion  and the speed of convergence for the stochastic Taylor expansion can be improved using Borel-Cantelli type arguments when $H \in \left( \frac{1}{2}, \frac{3}{4}  \right)$.
\end{abstract}

\baselineskip 0.25in

\tableofcontents

\newpage

\section{Introduction }
This paper is divided into two parts. In the first part, we consider a deterministic differential equation of the following type:

\begin{align}\label{de2}
\begin{cases}
   dX_{t}=\displaystyle \sum^{d}_{i=0}V_{i}\left(X_{t}\right)dy^{i}_{t} \\
   X_0=x_0
\end{cases}
 \end{align}

where the $V_i$'s are Lipschitz $C^{\infty}$ vector fields on $\mathbb{R}^n$ with Lipschitz derivatives
and where the driving signal $y:\mathbb{R}_{\ge 0} \rightarrow \mathbb{R}^{d+1}$ is $\beta$-H\"{o}lder continuous with $\beta>1/2$. Since $y:\mathbb{R}_{\ge 0} \rightarrow \mathbb{R}^{d+1}$ is $\beta$-H\"{o}lder with $\beta>1/2$, the integrals $\displaystyle \int^{t}_{0}V_{i}(X_s)dy^{i}_{t}$ are understood in Young's sense (see \cite{Young} and \cite{MZ}) and it is known that the above equation has a unique solution that can be constructed as a fixed point of the Picard iteration in a convenient Banach space of functions (see \cite{Nualart}, \cite{MZ} and \cite{Ruz}).

We  express, under suitable conditions, the solution of the differential equation as an infinite series involving the iterated integrals of the H\"{o}lder path. To study the question of the convergence  of the infinite series, which we will call the Taylor expansion associated with the differential equation later on, we will develop some convergence criteria based on an  explicit upper bound of the  iterated integrals of the H\"{o}lder path. The main tool we use to obtain such  bounds is the fractional calculus and the main assumption we need to perform the expansion is the analycity of the vector fields $V_i$'s in a neighborhood of the initial condition $x_0$. To illustrate our results, we  consider examples of differential equations on a Lie group.

The second part of this paper is concerned with a stochastic differential equation (in short SDE) of the following type:

\begin{align*}
\begin{cases}
   dX_{t}=\displaystyle \sum^{d}_{i=0}V_{i}\left(X_{t}\right)dB^{i}_{t} \\
   X_0=x_0
\end{cases}
 \end{align*}

where the $V_i$'s are Lipschitz $C^{\infty}$ vector fields on $\mathbb{R}^n$ with Lipschitz derivatives and where $(B_t)_{t \ge 0}$ is a $d$-dimensional fractional Brownian motion with Hurst parameter $H>1/2$. Throughout the paper, we use the convention that $B^0_t=t$. We will be interested in the convergence of Taylor expansion of the solution to the SDE. The deterministic estimates we obtained in the first part will automatically  lead to convergence criteria for the stochastic Taylor expansion. The rate of convergence  to $0$ and $L^p$ estimates  of the remainder term are then studied. Afterwards, inspired by some ideas of  Ben Arous \cite{G.Ben} (see also Castell \cite{Cast}), we will also use probabilistic methods to obtain a convergence criterion, which improves the convergence radius of the stochastic Taylor expansion in some cases. The key point  to our probabilistic methods are the   $L_{2}$ estimates of the iterated integrals of the fractional Brownian motion that were obtained by Baudoin-Coutin in  \cite{FB1} and Neuenkirch-Nourdin-R\"{o}{\ss}ler-Tindel in \cite{Tindel}.

\section{Taylor expansion for differential equations driven by H\"{o}lder paths}
In this section, we present our main theorem. We derive a convergence criterion, under which, we can, on a nonempty time interval,  write the solution of  our differential equation (\ref{de2}) as an infinite series of iterated integrals of  the H\"{o}lder path $y$. The materials in this section will be organized in the following way:

\begin{enumerate}
\item In Section 2.1, we will recall some basic facts about fractional calculus that will be needed in our analysis.
\item In Section  2.2, we will briefly discuss the existence and uniqueness of the solution to the differential equation (\ref{de2}). The main result is borrowed from \cite{Nualart}.  We will also define, in apparently two different ways, the Taylor expansion of the solution of the differential equation (\ref{de2}). The two apparently different expansions are then shown to eventually  be the same.
\item In Section 2.3, we will state our main theorem.
\item In Section 2.4 we will apply our results to the case of differential equations on Lie groups.
\end{enumerate}

\

\subsection{Preliminaries : Fractional calculus}
Let us first  do some remainders about fractional calculus. For further details, we refer the reader to (\cite{MZ}) or (\cite{Nualart}). Let $f\in L^{1}(a,b)$ and $\alpha>0$. The left-sided and right-sided fractional integrals of $f$ of order $\alpha$ are defined by:

\[
I^{\alpha}_{a+}f(x)=\frac{1}{\Gamma(\alpha)}\int^{x}_{a}(x-y)^{\alpha-1}f(y)dy
\]
and
\[
I^{\alpha}_{b-}f(x)=\frac{(-1)^{-\alpha}}{\Gamma(\alpha)}\int^{b}_{x}(y-x)^{\alpha-1}f(y)dy
\]

respectively, where $(-1)^{-\alpha}=e^{-i\pi\alpha}$ and $\Gamma(\alpha)=\int^{\infty}_{0}u^{\alpha-1}e^{-u}du$ is the Gamma function. Let us denote by $I^{\alpha}_{a+}(L^{p})$ (respectively $I^{\alpha}_{b-}(L^{p})$) the image of $L^{p}(a,b)$ by the operator $I^{\alpha}_{a+}$ (respectively $I^{\alpha}_{a+}$). If $f\in I^{\alpha}_{a+}(L^{p})$ (respectively  $f\in I^{\alpha}_{b-}(L^{p})$) and $0<\alpha<1$, we define for $x\in(a,b)$ the left and right  Weyl derivatives by:

\[
D^{\alpha}_{a+}f(x)=\frac{1}{\Gamma(1-\alpha)}\left(\frac{f(x)}{(x-a)^{\alpha}}+\alpha\int^{x}_{a}\frac{f(x)-f(y)}{(x-y)^{\alpha+1}}dy\right)\mathbf{1}_{(a,b)}(x)
\]

and respectively,
\[
D^{\alpha}_{b-}f(x)=\frac{(-1)^{\alpha}}{\Gamma(1-\alpha)}\left(\frac{f(x)}{(b-x)^{\alpha}}+\alpha\int^{b}_{x}\frac{f(x)-f(y)}{(y-x)^{\alpha+1}}dy\right)\mathbf{1}_{(a,b)}(x)
\]

Now recall that from (\cite{Nualart}), for a parameter $0<\alpha<1/2$, $W^{1-\alpha, \infty}_{T}(0,T)$ is defined as the space of measurable function $g: [0,T]\rightarrow \mathbb{R}$ such that:

\[
\|g\|_{1-\alpha,\infty,T}= \sup_{0<s<t<T}\left(\frac{|g(t)-g(s)|}{(t-s)^{1-\alpha}}+\int^{t}_{s}\frac{|g(y)-g(s)|}{(y-s)^{2-\alpha}}dy\right)<\infty
\]

Clearly,  for every  $ \epsilon>0 $,
\[
C^{1-\alpha+\epsilon}(0,T)\subset W^{1-\alpha, \infty}_{T}(0,T)\subset C^{1-\alpha}(0,T)
\]

Moreover, if $g\in W^{1-\alpha, \infty}_{T}(0,T)$, its restriction to $(0,t)$ belongs to $I^{1-\alpha}_{t-}(L^{\infty}(0,t))$ for every $t$ and

\begin{align*}
\Lambda_\alpha (g):=&\frac{1}{\Gamma(1-\alpha) } \sup_{0 <s<t<T} | (D^{1-\alpha}_{t-}g_{t-})(s) |.\\
               \leq& \frac{1}{\Gamma(1-\alpha)\Gamma(\alpha)}\|g\|_{1-\alpha,\infty,T}<\infty
\end{align*}

We also denote by $W^{\alpha,1}_{0}(0,T)$ the space of measurable functions f on $[0,T]$ such that:

\[
\|f\|_{\alpha,1}=\int^{T}_{0}\frac{f(s)}{s^{\alpha}}ds+\int^{T}_{0}\int^{s}_{0}\frac{|f(s)-f(y)|}{(s-y)^{\alpha+1}}dyds<\infty
\]

The restriction of $f\in W^{\alpha,1}_{0}(0,T)$ to $(0,t)$ belongs to $I^{\alpha}_{0+}(L^{1}(0,t))$ for all t.\\

\subsection{Taylor expansion of the solution}.

For the convenience of the reader, we first recall the main result concerning the existence and uniqueness results for solutions of the differential equation driven by H\"{o}lder path with H\"{o}lder exponent $\beta>1/2$. Several authors made contributions to this existence and uniqueness result. It has been first proved by M.Z\"{a}hle in \cite{MZ}. D. Nualart and A.Rascanu obtained the same result independently in \cite{Nualart}. Also, a different approach was introduced by A. A. Ruzimaikina to prove this existence and uniqueness theorem in \cite{Ruz}. The following result is due to Nualart and Rascanu (\cite{Nualart})

\begin{theorem}
Let $0<\alpha<1/2$ be fixed. Let $g\in W^{1-\alpha,\infty}(0,T;\mathbb{R}^{d})$. Consider the deterministic differential equation on $\mathbb{R}^{n}$:

\begin{align}\label{de1}
x^{i}_{t}=x^{i}_{0}+\int^{t}_{0}b^{i}(s,x_s)ds+\sum^{d}_{j=1}\int^{t}_{0}\sigma^{i,j}(s,x_s)dg^{j}_{s}, \quad  t\in[0,T]
\end{align}

$i=1,\cdots,n$ where $x_0\in \mathbb{R}^{n}$, and the coefficients $\sigma^{i,j}, b^{i}: [0,T]\times \mathbb{R}^{n}\rightarrow \mathbb{R}$ are measurable functions satisfying the following assumptions with $\rho=1/\alpha$, $0<\eta,\delta\leq 1$ and
\[
0<\alpha<\alpha_0=\min\left\{\frac{1}{2},\eta,\frac{\delta}{1+\delta}\right\}
\]

\begin{enumerate}
\item $\sigma(t,x)=(\sigma^{i,j}(t,x))_{n\times d}$ is differentiable in x, and there exist some constants $0<\eta,\delta \leq 1$ and for every $N\geq 0$ there exists $M_{N}>0$ such that the following properties hold:
 \begin{align}
 \begin{cases}
 \|\sigma(t,x)-\sigma(t,y)\|\leq M_{0}\|x-y\|, \quad x\in \mathbb{R}^{n}, \forall t\in [0,T]\\
 \\
 \|\partial_{x_i}\sigma(t,x)-\partial_{y_i}\sigma(t,y)\| \leq M_{N}\|x-y\|^{\delta} \quad \|x\|, \|y\| \leq N, \forall t\in [0,T]\\
 \\
 \|\sigma(t,x)-\sigma(t,y)\|+ \|\partial_{x_i}\sigma(t,x)-\partial_{y_i}\sigma(t,y)\| \leq M_{0}\|t-s\|^{\eta}, \forall t,s\in[0,T]
 \end{cases}
 \end{align}
\item There exists $b_0\in L^{\rho}(0,T;\mathbb{R}^{n})$, where $\rho\geq 2$, and for every $N\geq 0$ there exists $L_{N}>0$ such that the following properties hold:
 \begin{align}
 \begin{cases}
 \|b(t,x)-b(t,y)\|\leq L_{N}\|x-y\|, \quad \forall \|x\|, \|y\|\leq N, \forall t\in[0,T]\\
 \\
 \|b(t,x)\|\leq L_{0}\|x\|+b_{0}(t), \quad \forall x\in \mathbb{R}^{d}, \forall t\in[0,T]
 \end{cases}
\end{align}
\end{enumerate}

Then the differential equation (\ref{de1}) has a unique solution $x\in W^{\alpha,\infty}_{0}(0,T;\mathbb{R}^{n})$. Moreover, the solution $x$ is $(1-\alpha)$-H\"{o}lder continuous.

\end{theorem}

With this existence and uniqueness result in hands, we now come back to the differential equation we are concerned with:

\begin{align}\label{dep3}
\begin{cases}
   dX_{t}=\displaystyle \sum^{d}_{i=0}V_{i}\left(X_{t}\right)dy^{i}_{t} \\
   X_0=x_0
\end{cases}
 \end{align}

Throughout the section, we make the following assumptions:
\begin{enumerate}
\item
The vector fields $V_{i}'s$ are Lipschitz, $C^{\infty}$, with Lipschitz derivatives, and analytic on the set $\{x:  \|x-x_0\| \leq C\}$ for some $C>0$
\item  The driving path $y:\mathbb{R}_{\ge 0} \rightarrow \mathbb{R}^{d+1}$ is $\beta$-H\"{o}lder continuous with $\beta>1/2$.
\end{enumerate}
Clearly, from by Theorem 1, under our assumptions, this differential equation  admits a unique solution $(X_t)_{t \ge 0}$.

Let $f: \mathbb{R}^{n} \rightarrow \mathbb{R}$ be a $C^{\infty}$ function. By the change of variable formula, we have
\[
f(X_t)=f(x_0)+\sum^{d}_{i=0}\int^{t}_{0}V_{i}f(X_s)dy^{i}_{s}.
\]

Now, a new application of the  change of variable formula to $V_{i}f(X_s)$ leads to
\[
f(X_t)=f(x_0)+\sum^{d}_{i=0}V_{i}f(x_0)\int^{t}_{0}dy^{i}_{s}+\sum^{d}_{i,j=0}\int^{t}_{0}\int^{s}_{0}
V_{j}V_{i}f(X_u)dy^{j}_{u}dy^{i}_{s}.
\]

We can continue this procedure to get after $N$ steps
\[
f(X_t)=f(x_0)+\sum^{N}_{k=1}\sum_{I=(i_1,\cdots,i_k)}(V_{i_1}\cdots V_{i_k}f)(x_0)\int_{\triangle^{k}[0,t]}dy^{I}+R_{N}(t)
\]
for some remainder term $R_{N}(t)$, where we used the notations:
\begin{enumerate}
\item $\triangle^{k}[0,t]=\{(t_1,\cdots,t_k)\in[0,t]^{k}, 0\leq t_1\leq t_2\cdots\leq t_k\leq t\}$
\item If $I=\left(i_1,\cdots,i_k\right)\in\{0,1,\cdots,d\}^k$ is a word with length $k$, \[\int_{\triangle^{k}[0,t]}dy^{I}=\displaystyle
    \int_{0 \le t_1 \le t_2 \le \cdots \le t_k \le t}dy^{i_1}_{t_1}\cdots dy^{i_k}_{t_k}.
    \]

\end{enumerate}

If we let $N\rightarrow +\infty$, we are led to the formal expansion formula:

\[
f(X_t)=f(x_0)+\sum^{\infty}_{k=1}\sum_{I=(i_1,\cdots,i_k)}(V_{i_1}\cdots V_{i_k}f)(x_0)\int_{\triangle^{k}[0,t]}dy^{I}
\]

Now let us denote $\pi^{j}(x)=x_j$, the $j$-th projection map. By using the previous expansion with $f=\pi^{j}$ we get
\[
X^{j}_{t}=x^{j}_{0}+\sum^{\infty}_{k=1}\sum_{I=(i_1,\cdots,i_k)}(V_{i_1}\cdots V_{i_k}\pi^{j})(x_0)\int_{\triangle^{k}[0,t]}dy^{I}
\]

Therefore we have formally,
\[
X_t=x_{0}+\sum^{+\infty}_{k=1}g_{k}(t)
\]

where
\[
g^{j}_{k}(t)=\sum_{|I|=k}(V_{i_1}\cdots V_{i_k}\pi^{j})(x_0)\int_{\triangle^{k}[0,t]}dy^{I}.
\]
This leads to the following definition:

\begin{definition}
The Taylor expansion associated with the differential equation (\ref{dep3}) is defined as
\[
x_0+\sum^{\infty}_{k=1}g_{k}(t)
\]
where
\[
g^{j}_{k}(t)=\sum_{|I|=k}P^{j}_{I}\int_{\triangle^{k}[0,t]}dy^{I}, \quad P^{j}_{I}=(V_{i_1}\cdots V_{i_k}\pi^{j})(x_0).
\]
\end{definition}

Of course, at that point, the Taylor expansion is only a formal object in the sense that the convergence questions are not addressed yet. Our goal will be to provide natural assumptions ensuring the convergence of the expansion. Let us observe that the expansion can be written in a more compact way by using labeled rooted trees. For details on the combinatorics associated to these trees and the compact form of the expansion, we refer the interested reader to \cite{Tindel}.

Another form of the expansion, which will be convenient for us,  relies on  the Taylor expansion of the vector fields and can be obtained as described below. 
We first need to  introduce the following notations. 
Let $\displaystyle \sum_{\alpha\in N^{n}}b^{\alpha}_{i}\left(x-x_0\right)^{\alpha}$ be the formal Taylor series of $V_{i}$ at $x_0$ with the usual convention that $x^{\alpha}=x^{\alpha_1}_{1}\cdots x^{\alpha_n}_{n}$ if $\alpha=\left(\alpha_1,\cdots,\alpha_n\right)$ and $x=\left(x_1,\cdots,x_n\right)$. Let now  $C^{m}_{i}\left(h_1,\cdots,h_k\right)$ be the coefficient of $y^{m}$ in   $\displaystyle \sum_{\alpha\in N^{n}}b^{\alpha}_{i}\left(h_{1}y+\cdots+h_{k}y^{k}\right)^{\alpha}$. 

 Consider the following rescaled differential equation, which depends on the parameter $\epsilon$:

\begin{align}\label{pde}
\begin{cases}
   dX^{\epsilon}_{t}=\displaystyle \sum^{d}_{i=0}\epsilon V_{i}(X^{\epsilon}_{t})dy^{i}_{t}\\
   X^{\epsilon}_0=x_0.
\end{cases}
 \end{align}

Assume now that $X_t^\epsilon$ admits an expansion in powers of $\epsilon$ as  $x_0+\displaystyle \sum^{\infty}_{k=1}h_{k}(t)\epsilon^{k}$ for some $h_k$'s. Heuristically, because of the analyticity of the $V_{i}'s$, we can expand $V_{i}\left(x_0+\displaystyle \sum^{\infty}_{k=1}h_{k}(t)\epsilon^{k}\right)$ at $x_0$. Therefore, on the right hand side of (\ref{pde}), we have:

\begin{align*}
\sum^{d}_{i=0}\epsilon V_{i}(X^{\epsilon}_{t})dy^{i}_{t}
&=\sum^{d}_{i=0}\epsilon\left(V_{i}(x_0)+C^{1}_{i}(h_1)\epsilon+C^{2}_{i}\left(h_1,h_2\right)\epsilon^{2}+\cdots \right)dy^{i}_{t}\\
& =\left(\left(\sum^{d}_{i=0}V_{i}(x_0)dy^{i}_{t}\right)\epsilon+\left(\sum^{d}_{i=0}C^{1}_{i}(h_1)dy^{i}_{t}\right)\epsilon^{2}+\cdots\right)
\end{align*}

Therefore, by identifying this expression with the left side of the equation we are let with the conclusion that  the $h_{k}(t)'s$ need to satisfy the following inductive system of equations:
\begin{align}\label{iteration}
\begin{cases}
  dh_1(t)=\displaystyle \sum^{d}_{i=0}V_{i}(x_0) dy^{i}_{t} \\
  dh_2(t)=\displaystyle \sum^{d}_{i=0}C^{1}_{i}\left(h_{1}(t)\right)dy^{i}_{t}\\
  .\\
  .\\
  .\\
  dh_{k}(t)=\displaystyle \sum^{d}_{i=0}C^{k-1}_{i}\left(h_{1}(t),\cdots,h_{k-1}(t)\right)dy^{i}_{t}
\end{cases}
 \end{align}
 If we let $\epsilon=1$, we therefore see that $x_0+\displaystyle \sum^{\infty}_{k=1}h_{k}(t)$ formally solves the equation (\ref{dep3}) provided that the $h_k$'s satisfy the above inductive system.

We have the following result:
\begin{proposition}
Define $h_k$ inductively by the system (\ref{iteration}), then for every $k \in \mathbb{N}$ and $t \ge 0$ we have $h_{k}(t)=g_{k}(t)$.
\end{proposition}

\begin{proof}

First, let us observe that for every $t \ge 0$, $ \epsilon \to X^\epsilon_t$ is a smooth function, indeed the equation 
\[
\begin{cases}
   dX^{\epsilon}_{t}=\displaystyle \sum^{d}_{i=0}\epsilon V_{i}(X^{\epsilon}_{t})dy^{i}_{t}\\
   X^{\epsilon}_0=x_0.
\end{cases}
\]
can rewritten as
\[
\begin{cases}
   dY_{t}=\displaystyle \sum^{d}_{i=0} \tilde{V}_{i}(Y_{t})dy^{i}_{t}\\
   Y_0=(\epsilon, x_0),
\end{cases}
\]
here $Y_t=(\epsilon, X_t^\epsilon)$ is now valued in $\mathbb{R} \times \mathbb{R}^n$ and $\tilde{V_i}(\epsilon,x)=(0,\epsilon V_i(x))$.
The problem of the smoothness of $X_t^\epsilon$ with respect to $\epsilon$ is thus reduced to the problem of the smoothness of the solution of an equation driven by H\"older paths with respect to its initial condition. Theorems that apply in  this situation may be found in \cite{DN} and \cite{NS}.

From the definition of $h_{k}(t)$ which is given by the system (\ref{iteration}), it is then clear by induction that 
\[
h_k(t)=\frac{1}{k!} \frac{d^k X_t^{\epsilon}}{d\epsilon^k} (0).
\]

On the other hand, as we have seen before,  the repeated application of the change of  variable formula to the smooth function $f(x)=\pi^{j}(x)$, gives
\[
X^{\epsilon,j}_{t}=x^{j}_{0}+\sum^{N}_{k=1}\epsilon^{k}\sum_{I=(i_1,\cdots,i_k)}(V_{i_1}\cdots V_{i_k}\pi^{j})(x_0)\int_{\triangle^{k}[0,t]}dy^{I}+\epsilon^{N+1}R^{\epsilon,j}_{N}(t),
\]
where
\[
R^{\epsilon,j}_{N}(t)=\sum_{I=(i_1,\cdots, i_{N+1})} \int_{\Delta^{N+1}[0,t]} (V_{i_1} \cdots V_{i_{N+1}} \pi^j) (X_{s_1}^\epsilon) dy^{i_1}_{s_1}\cdots dy^{i_{N+1}}_{s_{N+1}}.
\]
Therefore we have
\[
X^{\epsilon}_{t}=x_0+\sum^{N}_{k=1}g_{k}(t)\epsilon^{k}+\epsilon^{N+1}R^\epsilon_{N}(t).
\]
Computing $\frac{d^k X_t^{\epsilon}}{d\epsilon^k} (0)$ by using the previous expression immediately gives $h_{k}(t)=g_{k}(t)$.
\end{proof}

\subsection{Convergence of the Taylor expansion}

We now address the convergence questions related to the Taylor expansion. We use the same assumptions and the same notations as in the previous section. In particular we remind that the vector fields $V_{i}'s$ are Lipschitz, $C^{\infty}$ with Lipschitz derivatives, and analytic on the set $\{x:  \|x-x_0\| \leq C\}$ for some $C>0$.

\subsubsection{A general result}

We first have the following general result. 

\begin{theorem}\label{general}
Let $x_0+\sum^{\infty}_{k=1}g_{k}(t)$ be the Taylor expansion associated with the equation (\ref{dep3}). There exists $T>0$ such that for $0\le t <T$ the series
\[
 \sum^{\infty}_{k=1}\| g_{k}(t)\|
\]
is convergent and 
\[
X_t=x_0+ \sum^{\infty}_{k=1}g_{k}(t).
\]
\end{theorem}

\begin{proof}

Let us fix $\rho >0$. For $z \in \mathbb{C}$, $ |z| <\rho $, we consider the complex differential equation
\[
\begin{cases}
   dX^{z}_{t}=\displaystyle \sum^{d}_{i=0}z V_{i}(X^{z}_{t})dy^{i}_{t}\\
   X^{z}_0=x_0,
\end{cases}
\]
which, by analycity of the $V_i$'s,  is well defined at least up to the positive time
\[
T(\rho)=\inf_{z, |z| <\rho} \inf \{ t \ge 0, X_t^z \notin B(x_0,C) \}.
\]
As we have seen before, for a fixed $t \ge 0$, the map $z \to X_t^z$ is seen to be smooth. We claim it is even analytic. Indeed, differentiating with respect to $\bar{z}$ the integral expression
\[
X^{z}_{t}=x_0+\int_0^t  \sum^{d}_{i=0}z V_{i}(X^{z}_{s})dy^{i}_{s},
\]
immediately gives $\frac{\partial X_t^z}{\partial \bar{z}}=0$. That is $z \to X_t^z$ is analytic on the disc  $ |z| <\rho $. 

On the other hand, $x_0+\sum^{\infty}_{k=1}z^k g_{k}(t)$ is precisely the Taylor expansion of $X_t^z$. Choosing then $\rho >1$, finishes the proof.
\end{proof}

\subsubsection{Quantitative bounds}

Theorem \ref{general} is very general but gives few informations concerning the radius of convergence or the  speed of convergence of the Taylor expansion 

\[
x_0+ \sum^{\infty}_{k=1}g_{k}(t).
\]
In this section, we shall be interested in more quantitative bounds.
\begin{lemma}\label{convergence}

Let $x_0+\displaystyle \sum^{\infty}_{k=1}g_{k}(t)$ be the Taylor expansion associated with the equation (\ref{dep3}). For $ r>1$, we define
\[
 T_{C}(r)=\inf\left(t,~\displaystyle \sum^{\infty}_{k=1}r^{k}\|g_{k}(t)\| \geq C\right).
\]
For every $t \le T_C(r)$ we have
\[
X_t=x_0+\displaystyle \sum^{\infty}_{k=1}g_{k}(t).
\]
\end{lemma}
\begin{remark}
By using a similar argument as in the proof of Theorem \ref{general},  it is seen that in the previous theorem, we always have   $T_C(r)>0$. 
\end{remark}
\begin{proof}
Consider the following functions:
\[
G_{i}(\epsilon)=V_{i}\left(x_0+\displaystyle \sum_{k=1}^{\infty}\epsilon^{k}g_{k}(t)\right)
\]
\noindent Then by the definition of $T_{C}(r)$, if $t<T_{C}(r)$, $G_{i}(\epsilon)$ is analytic on the disc $\{\epsilon\in \mathbb{C}:~~\|\epsilon\|< r\}$. Hence for any $r'<r$ and if $\|\epsilon\|<r'$, we have:
\begin{align}
&\left\|G_{i}(\epsilon)-(V_{i}(x_0)+\displaystyle\sum^{m-1}_{k=1}C^{k}_{i}(g_{1}(t),g_{2}(t),..,g_{k}(t)))\epsilon^{k}\right\|\\
\leq & M^{i}_{r'}.\displaystyle \sum^{\infty}_{k=m}\left(\frac{\epsilon}{r'}\right)^{m}\\
\leq & K^{i}. R_{m} \label{remainder}
\end{align}

 where $M^{i}_{r'}=\displaystyle \sup_{\|\epsilon\|\leq r'}G_{i}(\epsilon)\leq \sup_{\|x\|\leq C}\|V_{i}(x+x_0)\|=K^{i}$ And $R_{m}$ denotes the tail of the convergent geometric series $\displaystyle \sum^{\infty}_{k=1}\left(\frac{\epsilon}{r'}\right)^{k}$. We know that $\displaystyle R_{m}\longrightarrow 0, as~~ m \longrightarrow +\infty$. Now by the definition of $g_{k}(t)$ we write:

\begin{align*}
\sum^{m}_{k=1}\epsilon^{k}g_{k}(t)=&\sum^{d}_{i=0}\sum^{m}_{k=1}\epsilon^{k}\int^{t}_{0}C^{k-1}_{i}\left(g_{1}(s),g_{2}(s),...,g_{k-1}(s)\right)dy^{i}_{s}\\
                                  =&\sum^{d}_{i=0}\epsilon\sum^{m}_{k=1}\int^{t}_{0}\left(\epsilon^{k-1}C^{k-1}_{i}\left(g_{1}(s),g_{2}(s),...,g_{k-1}(s)\right)\right)dy^{i}_{s}\\
                                  =&\sum^{d}_{i=0}\epsilon\int^{t}_{0}\left(V_{i}(x_0)+\sum^{m-1}_{k=1}\epsilon^{k}C^{k}_{i}\left(g_{1}(s),....,g_{k}(s)\right)\right)dy^{i}_{s}
\end{align*}

 Since when $t<T_{C}(r)$, the series $\displaystyle \sum^{\infty}_{k=1}\epsilon^{k}g_{k}(t)\leq C$ ,we replace t by $t\wedge T_{C}(r)$ in the previous equation:

\begin{align}\label{poly id}
\sum^{m}_{k=1}\epsilon^{k}g_{k}\left(t\wedge T_{C}(r)\right)=\sum^{d}_{i=0}\epsilon\int^{t\wedge T_{C}(r)}_{0}\left(V_{i}(x_0)+\sum^{m-1}_{k=1}\epsilon^{k}C^{k}_{i}\left(g_{1}(s),....,g_{k}(s)\right)\right)dy^{i}_{s}
\end{align}

 with the estimate (\ref{remainder}), we have when $\epsilon < r'$,

\begin{align}\label{estmate}
\left\|G_{i}\left(\epsilon\right)-\left(V_{i}(x_0)+\displaystyle\sum^{m-1}_{k=1}C^{k}_{i}\left(g_{1}(t),g_{2}(t),....g_{k}(t)\right)\epsilon^{k}\right)\right\| \leq K^{i}R_{m}
\end{align}

Therefore, since $r>1$ we can choose $1<r'<r$ and let $m\rightarrow \infty$, then the right hand side of (\ref{estmate}) tends to 0. Therefore,  by letting $m\rightarrow \infty$ in equation (\ref{poly id}) we have:
\begin{align*}
x_0+\sum^{\infty}_{k=1}\epsilon^{k}g_{k}(t\wedge T_{C}(r))
=&x_0+\sum^{d}_{i=0}\int^{t\wedge T_{C}(r)}_{0}\epsilon G_{i}(t)dy^{i}_s\\
=&x_0+\sum^{d}_{i=0}\int^{t\wedge T_{C}(r)}_{0}\epsilon V_{i}\left(x_0+\sum^{\infty}_{k=1}\epsilon^{k}g_{k}(s)\right)dy^{i}_s
\end{align*}

And the result follows when we take $\epsilon =1$
\end{proof}

The crucial estimates we need to control the speed of convergence of the Taylor expansion are provided by the following proposition.

\begin{proposition}\label{iterated}
Let $\alpha$ such that $1-\beta <\alpha <\frac{1}{2}$ and  fix $T>0$. For $0<t\leq T$, we have
\[
\left|  \int_{0 \leq t_1 \leq ... \leq t_n \leq t}
dy^{i_1}_{t_1}  \cdots  dy^{i_n}_{t_n} \right| \le\frac{\Gamma(1-2\alpha)}{\Gamma(n(1-2\alpha)) } C_{\alpha}(T)^{n-1} \Lambda_\alpha (T,y)^{n-1} \| y \|_{\alpha,T,\infty}
\]
where
\begin{align*}
&\Lambda_\alpha (T,y):=\frac{1}{\Gamma(1-\alpha) } \sup_{0 <s<t<T} \| (D^{1-\alpha}_{t-}y_{t-})(s) \|\\
&C_{\alpha}(T)=\left(\frac{1}{\alpha(1-\alpha)}+T^{\alpha}\right)\left(1+T^{\alpha}\right)\Gamma(1-2\alpha)T^{1-2\alpha}\\
&\| y \|_{\alpha,T,\infty} =\sup_{0 \le t \le T} \left( \| y(t) \| + \int_0^t \frac{\| y(t) -y(s) \| }{(t-s)^{1+\alpha} } ds \right).\
\end{align*}

\end{proposition}

\begin{proof}
Since $y: \mathbb{R}_{\ge 0} \rightarrow \mathbb{R}^{d+1}$ is a H\"{o}lder path with exponent $\beta>1/2$, and since $\alpha +\beta >1$, we have $y^{i}_{t}\in W^{1-\alpha,\infty}_{T}(0,T)$ for $i=0,1,2\cdots ,d $. For $t\leq T$, we denote $\displaystyle I_k(t)= \int_{0 \leq t_1 \leq ... \leq t_k \leq t}dy^{i_1}_{t_1}  \cdots  dy^{i_k}_{t_k}$, $k=1,\cdots,n$. Clearly, $I_{k}(t)\in W^{\alpha,1}_{0}(0,T)$ for $k=1,2,\cdots,n$. Therefore, Proposition 4.1 in Nualart-Rascanu (\cite{Nualart}) gives
\begin{align*}
&\left| I_{k+1}(t) \right| +\int_0^t \frac{\left|I_{k+1}(t) -I_{k+1}(s) \right| }{(t-s)^{1+\alpha} }ds\\
\le & \Lambda_\alpha (T,y) \left( \frac{1}{\alpha(1-\alpha)} +T^\alpha \right) \int_0^t \left( (t-r_{k})^{-2\alpha}+r^{-\alpha}_{k} \right) \left( \left| I_{k}(r_{k}) \right| +\int_0^{r_{k}} \frac{\left|I_{k}(r_{k}) -I_{k}(s) \right| }{(r_{k}-s)^{1+\alpha} }ds \right)dr_{k}.
\end{align*}
Using now the basic estimate
\[
(t-r)^{-2\alpha}+r^{-\alpha} \le \left( 1+T^\alpha \right) t^{2\alpha} (t-r)^{-2\alpha} r^{-2\alpha},
\]

we get
\begin{align*}
&\left| I_{k+1}(t) \right| +\int_0^t \frac{\left|I_{k+1}(t) -I_{k+1}(s) \right| }{(t-s)^{1+\alpha} }ds\\
\le & \Lambda_\alpha (T,y) \left( \frac{1}{\alpha(1-\alpha)} +T^\alpha \right)\left(1+T^{\alpha}\right) \int_0^t \left( t^{2\alpha}(t-r_{k})^{-2\alpha}r^{-2\alpha}_{k} \right) \left( \left| I_{k}(r_{k}) \right| +\int_0^{r_{k}} \frac{\left|I_{k}(r_{k}) -I_{k}(s) \right| }{(r_{k}-s)^{1+\alpha} }ds \right)dr_{k}.
\end{align*}
Iterating this inequality $k$ times, we get:
\begin{align*}
&\left| I_{k+1}(t) \right| +\int_0^t \frac{\left|I_{k+1}(t) -I_{k+1}(s) \right| }{(t-s)^{1+\alpha} }ds\\
& \le \Lambda_\alpha (T,y)^{k} \left( \frac{1}{\alpha(1-\alpha)} +T^\alpha \right)^{k}\left(1+T^{\alpha}\right)^{k}\|y\|_{\alpha,T,\infty} \int_0^t \left( t^{2\alpha}(t-r_{k})^{-2\alpha}r^{-2\alpha}_{k} \right)\int_0^{r_k} \left( r_{k}^{2\alpha}(r_{k}-r_{k-1})^{-2\alpha}r^{-2\alpha}_{k-1} \right)\\
&\int_0^{r_{k-1}} \left( r_{k-1}^{2\alpha}(r_{k-1}-r_{k-2})^{-2\alpha}r^{-2\alpha}_{k-2} \right)
\cdots \int_0^{r_2} \left( r_{2}^{2\alpha}(r_{2}-r_{1})^{-2\alpha}r^{-2\alpha}_{1} \right)dr_{1}dr_{2}\cdots dr_{k}\\
& \le \Lambda_\alpha (T,y)^{k} \left( \frac{1}{\alpha(1-\alpha)} +T^\alpha \right)^{k}\left(1+T^{\alpha}\right)^{k}\|y\|_{\alpha,T,\infty} \int_0^t \left( t^{2\alpha}(t-r_{k})^{-2\alpha}\right)\int_0^{r_k} \left(r_{k}-r_{k-1}\right)^{-2\alpha} \int_0^{r_{k-1}} \left(r_{k-1}-r_{k-2}\right)^{-2\alpha}\\
&\cdots \int_0^{r_2}\left(r_{2}-r_{1}\right)^{-2\alpha}r^{-2\alpha}_{1}dr_{1}dr_{2}\cdots dr_{k}
\end{align*}

Note that by the substitution $u=\displaystyle \frac{r_1}{r_2}$, we have:
\begin{align*}
\int_0^{r_2}\left(r_{2}-r_{1}\right)^{-2\alpha}r^{-2\alpha}_{1}dr_{1}=r_{2}^{1-4\alpha}
\int_0^{1}\left(1-u\right)^{-2\alpha}u^{-2\alpha}du=r_{2}^{1-4\alpha}B(1-2\alpha, 1-2\alpha)
\end{align*}

where
\[
B(x,y)=\frac{\Gamma(x)\Gamma(y)}{\Gamma(x+y)}=\int_0^{1}\left(1-u\right)^{x-1}u^{y-1}du
\]
is the Beta function. \\
Then plug $r_{2}^{1-4\alpha}B(1-2\alpha, 1-2\alpha)$ into the previous integrand and iterate this process, we get:
\begin{align*}
&\left| I_{k+1}(t) \right| +\int_0^t \frac{\left|I_{k+1}(t) -I_{k+1}(s) \right| }{(t-s)^{1+\alpha} }ds\\
& \le \Lambda_\alpha (T,y)^{k} \left( \frac{1}{\alpha(1-\alpha)} +T^\alpha \right)^{k}\left(1+T^{\alpha}\right)^{k}\|y\|_{\alpha,T,\infty}(T^{1-2\alpha})^{k}\prod^{k}_{i=1}B(1-2\alpha,i(1-2\alpha))\\
&=\Lambda_\alpha (T,y)^{k} \left( \frac{1}{\alpha(1-\alpha)} +T^\alpha \right)^{k}\left(1+T^{\alpha}\right)^{k}\|y\|_{\alpha,T,\infty}(T^{1-2\alpha})^{k}\prod^{k}_{i=1}\frac{\Gamma(1-2\alpha)\Gamma(i(1-2\alpha))}
{\Gamma((i+1)(1-2\alpha))}\\
&=\frac{\Gamma(1-2\alpha)}{\Gamma(k+1)(1-2\alpha)}\Lambda_\alpha (T,y)^{k} \left(\frac{1}{\alpha(1-\alpha)} +T^\alpha \right)^{k}\left(1+T^{\alpha}\right)^{k}\Gamma(1-2\alpha)^{k}(T^{1-2\alpha})^{k}\|y\|_{\alpha,T,\infty}
\end{align*}
\end{proof}

We are now in position  to give an estimate of the speed of   convergence of the stochastic Taylor expansion. The following lemma is taylor made for that purpose:

\begin{lemma}\label{bound}
Let $\alpha$ such that $1-\beta <\alpha <\frac{1}{2}$ and $\gamma<1-2\alpha$. There exists a constant $K_{\alpha,\gamma} >0$ such that for every $N \ge 0$ and $ x \ge 0$,
\[
\sum_{k=N+1}^{+\infty} \frac{\Gamma(k \gamma)}{\Gamma(k(1-2\alpha))} x^{k-1} \le K_{\alpha,\gamma} 
\begin{cases}
e^{2 x^{\frac{1}{1-2\alpha-\gamma}}} , \quad \text{if } N = 0 \\
 \frac{x^N e^{2 x^{\frac{1}{1-2\alpha-\gamma}}}}{\Gamma((1-2\alpha-\gamma)N)}, \quad \text{if } N \ge 1
\end{cases}
\]
\end{lemma}

\begin{proof}
We make the proof for $N \ge 1 $ and let the reader adapt the argument when $N=0$. We have
\begin{align*}
\sum_{k=N+1}^{+\infty} \frac{\Gamma(k \gamma)}{\Gamma(k(1-2\alpha))} x^{k-1} &=x^N \sum_{k=0}^{+\infty} \frac{\Gamma((k+N+1) \gamma)}{\Gamma((k+N+1)(1-2\alpha))} x^{k} \\
 & =\frac{x^N}{\Gamma((1-2\alpha-\gamma) N)} \sum_{k=0}^{+\infty} \frac{\Gamma((k+N+1) \gamma)\Gamma((1-2\alpha-\gamma)N)}{\Gamma((k+N+1)(1-2\alpha))} x^{k}  \\
 &\le C_{\alpha,\gamma} \frac{x^N}{\Gamma((1-2\alpha-\gamma)N)} \sum_{k=0}^{+\infty} \frac{x^k}{\Gamma( (k+1)(1-2\alpha-\gamma)) }
\end{align*}

From Lemma 7.7 in \cite{ Nualart}, we have for every $x \ge 0$.
\[
\sum^{\infty}_{k=0}\frac{x^{k}}{\Gamma((1+k)(1-2\alpha-\gamma))}\leq \frac{4e^{2}}{1-2\alpha-\gamma}e^{2x^{\frac{1}{1-2\alpha-\gamma}}}.
\]
This concludes the proof.
\end{proof}

We can now give convenient assumptions to decide the speed of convergence of the Taylor expansion. We remind the reader that if $I=(i_1,\cdots,i_k)$, we denote
\[
 P^{j}_{I}=(V_{i_1}\cdots V_{i_k}\pi^{j})(x_0).
 \]
 We also use the notations introduced in Proposition \ref{iterated}.
\begin{theorem}\label{convergence}
Let $\alpha$ such that $1-\beta <\alpha <\frac{1}{2}$. Let us assume that there exist $M>0$ and $0<\gamma<1-2\alpha$ such that for every word $I \in \{ 0,\cdots, d\}^k$
\begin{align}\label{criteria}
\|P_{I}\|\leq \Gamma(\gamma|I|)M^{|I|}.
\end{align}
Then, for every $r >1$, $T_C(r)>0$ and there exists a constant $ K_{\alpha,\gamma,M,d} >0 $ depending only on the subscript variables such that for every $0 \le  t <T_C(r)$ and $N \ge 1$,
\begin{align}\label{C1}
\left\|X_t-\left(x_0+\sum^{N}_{k=1}g_{k}(t)\right) \right\|\leq K_{\alpha,\gamma,M,d}  \| y \|_{\alpha,t, \infty}  \frac{(M(d+1)\Lambda_\alpha(t,y) C_\alpha(t) )^N }{\Gamma((1-2\alpha-\gamma)N)} e^{2 (M(d+1)\Lambda_\alpha(t,y) C_\alpha(t) )^{\frac{1}{1-2\alpha-\gamma}}}
\end{align}
In particular for every $0 \le  t <T_C(r)$,
\begin{align}\label{C2}
X_t=x_0+\sum^{+\infty}_{k=1}g_{k}(t).
\end{align}
Moreover if $C=+\infty$, that is if the $V_i$'s are analytic on $\mathbb{R}^n$, then (\ref{C1}) and (\ref{C2}) hold for every $t \ge 0$.
\end{theorem}
\begin{proof}
Let us fix $T>0$ and $r >1$. For every $ t\in [0,T]$, by Proposition \ref{iterated} we have:
\begin{align*}
&\sum^{\infty}_{k=1}r^{k}\|g_{k}(t)\|\\
&\leq \sum^{\infty}_{k=1}r^{k}\sum_{|I|=k}\|P_{I}\|\left|\int_{\triangle^{k}[0,t]}dy^{I}\right|\\
&\leq\sum^{\infty}_{k=1}r^{k}(d+1)^{k}M^{k}\Gamma(k\gamma)\Lambda_{\alpha}(T,y)^{k-1}C_{\alpha}(T)^{k-1}\|y\|_{\alpha,T,\infty}
\frac{\Gamma(1-2\alpha)}{\Gamma(k(1-2\alpha))}\\
&\leq\|y\|_{\alpha,T,\infty}r(d+1)M\Gamma(1-2\alpha)\sum^{\infty}_{k=1}\left(r(d+1)C_{\alpha}(T)\Lambda_{\alpha}(T,y)\right)^{k-1}
\frac{\Gamma(k\gamma)}{\Gamma(k(1-2\alpha))}\\
&=\|y\|_{\alpha,T,\infty}r(d+1)M\Gamma(1-2\alpha)\sum^{\infty}_{k=1}
\frac{\left(r(d+1)C_{\alpha}(T)\Lambda_{\alpha}(T,y)\right)^{k-1}}{\Gamma(k(1-2\alpha-\gamma))}B(k\gamma,k(1-2\alpha-\gamma))\\
&\leq B(\gamma,1-2\alpha-\gamma)\|y\|_{\alpha,T,\infty}r(d+1)M\Gamma(1-2\alpha)\sum^{\infty}_{k=1}
\frac{\left(r(d+1)C_{\alpha}(T)\Lambda_{\alpha}(T,y)\right)^{k-1}}{\Gamma(k(1-2\alpha-\gamma))}
\end{align*}
Now note that by our assumption, $1-2\alpha-\gamma>0$, then the series on the right hand side of our inequality converges. This implies that
\[
 T_{C}(r)=\inf\left(t,~\displaystyle \sum^{\infty}_{k=1}r^{k}\|g_{k}(t)\| \geq C\right) >0.
\]
We conclude from Lemma \ref{convergence} that for every $ 0 \le t < T_C(r)$, 
\[
X_t=x_0+\sum^{+\infty}_{k=1}g_{k}(t).
\]
If  moreover $C=+\infty$, that is the radius of convergence of the Taylor expansion of the $V_i$'s is $+\infty$, the previous equality holds for every $t \ge 0$ because, $\lim_{C \to +\infty} T_C(r) =+\infty$.

We then get for $0 \le t <T_C(r)$, thanks to  Lemma \ref{bound}:

\begin{align*}
&\left\|X_t-\left(x_0+\sum^{N}_{k=1}g_{k}(t)\right) \right\|\\
=&\left\|\sum^{\infty}_{k=N+1}g_{k}(t)\right\|\\
\leq& \sum^{\infty}_{k=N+1}\sum_{|I|=k}\|P_{I}\|\left|\int_{\triangle^{k}[0,t]}dy^{I}\right|\\
\leq&\sum^{\infty}_{k=N+1}(d+1)^{k}M^{k}\Gamma(k\gamma)\Lambda_{\alpha}(t,y)^{k-1}C_{\alpha}(T)^{k-1}\|y\|_{\alpha,t,\infty}
\frac{\Gamma(1-2\alpha)}{\Gamma(k(1-2\alpha))} \\
\le & K_{\alpha,\gamma,M,d}  \| y \|_{\alpha,t, \infty}  \frac{(M(d+1)\Lambda_\alpha(t,y) C_\alpha(t) )^N }{\Gamma((1-2\alpha-\gamma)N)} e^{2 (M(d+1)\Lambda_\alpha(t,y) C_\alpha(t) )^{\frac{1}{1-2\alpha-\gamma}}}
\end{align*}
\end{proof}

\subsection{The Lie group example}

In this section, we will see an application of our main results in the context of Lie group. Let us denote by $\bG$ a connected Lie group. For any $g\in \bG$, we have left and right translation $L_g$ and $R_g$, which are diffeomorphisms on $\bG$, defined by $L_{g}(h)=gh$ and $R_{g}(h)=hg$ for any $h\in G$. We say that a vector field $V$ is left-invariant on $\bG$ if for any $g\in \bG$,

\[
(L_g)_{*}(V)=V
\]

Let us denote by $\mathfrak{g}$ the Lie algebra of $\mathbb{G}$, that is the set of left-invariant vector fields on $\bG$ and by $T_{e}\bG$ the tangent space of $\bG$ at identity. We know that the map $V\rightarrow V(e)$ provides a linear isomorphism between $\mathfrak{g}$ and $T_{e}\bG$. Now let $V_{1},\cdots,V_{d} \in \mathfrak{g}$ be left-invariant vector fields and let us consider the following differential equation on Lie group $\bG$,

\begin{align}\label{de lie}
\begin{cases}
   dX_{t}=\displaystyle \sum^{d}_{i=0}V_{i}(X_{t})dy^{i}_{t} \\
   X_0=e
\end{cases}
\end{align}
where $y: \mathbb{R}_{\ge 0}\rightarrow \mathbb{R}^{d+1}$ is a $\beta$-H\"{o}lder path with $\beta>1/2$.  The vector fields are not supposed to be globally Lipschitz, so that we are not exactly in the previous framework. However, the vector fields $V_i$'s are smooth and analytic and the previous equation always admits a solution up to a  possible positive explosion time $\tau$.

\begin{remark}
If  $\bG=GL(n,\mathbb{R})$, we know the associated Lie algebra $\mathfrak{g}=M(n,\mathbb{R})$. In this case, the Taylor expansion associated with (\ref{de lie}) is convergent for every $t \ge 0$. We actually have $\gamma=0$ in the convergence criterion of Theorem \ref{convergence}.
\end{remark}

Before we state our main theorem in that context, let us introduce some notations (see \cite{FBb} and \cite{FB2} for further details).
\begin{itemize}
\item If $I=(i_1,\cdots,i_k)$ we denote by $V_I$ the Lie  commutator defined by
\[
V_I = [V_{i_1},[V_{i_2},...,[V_{i_{k-1}}, V_{i_{k}}]...].
\]

\item $\mathcal{S}_k$ is the set of the permutations of
$\{0,...,k\}$;
\item If $\sigma \in \mathcal{S}_k$, $e(\sigma)$ is the cardinality of the set
\[
\{ j \in \{1,...,k-1 \} , \sigma (j) > \sigma(j+1) \},
\]
\item
\[
\Lambda_I (y)_t= \sum_{\sigma \in \mathcal{S}_k} \frac{\left(
-1\right) ^{e(\sigma )}}{k^{2}\left(
\begin{array}{l}
k-1 \\
e(\sigma )
\end{array}
\right) } \int_{0 \leq t_1 \leq ... \leq t_k \leq t}
dy^{\sigma^{-1}(i_1)}_{t_1}  \cdots
dy^{\sigma^{-1}(i_k)}_{t_k},\quad t \ge 0.
\]
\end{itemize}

The main theorem of this section is the following:

\begin{theorem}
There exists $T>0$ such that for $0 \le t \le T$,
\[
X_t=\exp \left( \sum^{\infty}_{k=1}\sum_{(i_1,...i_k)}\Lambda_{I}(y)_{t}V_{I} \right),
\]
is the solution of the equation \ref{de lie}.
\end{theorem}

\begin{proof}
Let us denote by $\mathbf{ad}$ the map $\left(\mathbf{ad}x\right)(z)=[x,z]$. Consider the set $\mathfrak{g}_{e}=\{x\in T_{e}G~~,~~s(x)\cap2\pi K\subset\{0\}\}$, where $s(x)$ denotes the spectrum of $\mathbf{ad}x$. It is known that $\mathfrak{g}_{e}$ is an open set of $T_{e}G$ and $\forall x\in \mathfrak{g}_{e}$, $\displaystyle\left( \frac{I-e^{-\mathbf{ad}x}}{\mathbf{ad}x}\right)$ is invertible, where $\displaystyle \left(\frac{I-e^{-\mathbf{ad}x}}{\mathbf{ad}x}\right)$ is defined as power series $\displaystyle\sum^{\infty}_{k=0}\frac{1}{\left(k+1\right)!}\left(\mathbf{-ad}x\right)^{k}$. Now we define vector fields $U_{i}$ on $\mathfrak{g}_{e}$ by $\displaystyle U_{i}(x)=\left(\frac{\mathbf{ad}x}{I-e^{-\mathbf{ad}x}}\right)(V_{i}(e))$ where $\displaystyle\left( \frac{\mathbf{ad}x}{I-e^{-\mathbf{ad}x}}\right)$ denotes the inverse of $\displaystyle\left(\frac{I-e^{-\mathbf{ad}x}}{\mathbf{ad}x}\right)$\\

 We claim that for $\forall x \in \mathfrak{g}_{e}$ we have:
\[
(T_{x}exp)[U_{i}(x)]=V_{i}(expx)
\]
 In fact, since $V_{i}'s$ are left-invariant vector field on $\bG$, we have

\begin{align*}
\left(T_{x}exp\right)[U_{i}(x)]=&(L(expx))_{*}\circ \left(\frac{I-e^{-\mathbf{ad}x}}{\mathbf{ad}x}\right)[U_{i}(x)]\\
                    =&(L(expx))_{*}\circ \left(\frac{I-e^{-\mathbf{ad}x}}{\mathbf{ad}x}\right)\left(\frac{\mathbf{ad}x}{I-e^{-\mathbf{ad}x}}\right)(V_{i}(e))\\
                    =&\left(L(expx)\right)_{*}(V_{i}(e))\\
                    =&V_{i}(expx)
\end{align*}

Now if we consider the differential equation on $\mathfrak{g}_{e}$:
\begin{align}
\begin{cases}\label{de4}
   dY_{t}=\displaystyle \sum^{d}_{i=1}U_{i}(Y_{t})dy^{i}_{t} \\
   Y_0=0
\end{cases}
\end{align}

we know that $U_{i}(x)$ are analytic on $\mathfrak{g}_{e}$. Therefore from Theorem \ref{general}, there exists a positive time $T$ such that for $0\le t <T$,
\[
Y_t= \sum_{k=1}^\infty g_k(t),
\]
where   $\sum_{k=1}^\infty g_k(t)$ is the Taylor expansion of the equation defining $Y_t$. It is proved as in \cite{FBb} and \cite{Str} that
\[
g_k(t)=\sum_{(i_1,...i_k)}\Lambda_{I}(y)_{t}V_{I} .
\]

Now we claim that by the chain rule, we have $X_t=exp(Y_t)$ when $t<T$. In fact:
\begin{align*}
dX_t=&(T_{Y_{t}}exp)dY_{t}\\
    =&(T_{Y_{t}}exp)[\sum^{d}_{i=1}U_{i}(Y_{t})]dy^{i}_{t}\\
    =&\sum^{d}_{i=0}T_{Y_{t}}exp[U_{i}(Y_t)]dy^{i}_{t}\\
    =&\sum^{d}_{i=0}V_{i}(exp(Y_t))dy^{i}_{t}\\
    =&\sum^{d}_{i=0}V_{i}(X_t)dy^{i}_{t}
\end{align*}

with $X_{0}=exp(Y_{0})=e$

\end{proof}

\begin{remark}
It seems difficult to obtain precise universal bounds on the radius of convergence of $\sum^{\infty}_{k=1}\sum_{(i_1,...i_k)}\Lambda_{I}(y)_{t}V_{I}$ as it is the case for the radius of convergence of the Magnus expansion of absolutely continuous paths. Indeed, rough estimates lead to 
\[
 \sum_{\sigma \in \mathcal{S}_k} \frac{1}{k^{2}\left(
\begin{array}{l}
k-1 \\
e(\sigma )
\end{array}
\right) } \le \frac{C}{2^k} k! \sqrt{k},
\]
for some constant $C>0$.
We get therefore the bound
\[
\left| \Lambda_I (y)_t \right| \le \frac{C}{2^k} k! \sqrt{k} \frac{\Gamma(1-2\alpha)}{\Gamma(k(1-2\alpha)) } C_\alpha(t)^{k-1} \Lambda_\alpha (t, y)^{k-1} \| y \|_{\alpha,\infty},
\]
which unfortunately gives a convergence for the series for $\alpha=0$, which corresponds to the case where $y$ is absolutely continuous.
\end{remark}

\section{Taylor expansion for stochastic differential equations driven by fractional Brownian motions}

We now turn to the second part of this work and applies our result to stochastic differential equations driven by factional Brownian motion paths.

Let us recall that a $d$-dimensional fractional Brownian motion with Hurst parameter
$H \in (0,1)$ is a Gaussian process
\[
B_t = (B_t^1,\cdots,B_t^d), \text{ } t \geq 0,
\]
where $B^1,\cdots,B^d$ are $d$ independent centered Gaussian
processes with covariance function
\[
R\left( t,s\right) =\frac{1}{2}\left(
s^{2H}+t^{2H}-|t-s|^{2H}\right).
\]
It can be shown that such a process admits a continuous version
whose paths have finite $p$ variation for $1/p<H$. Let us observe
that for $H= \frac{1}{2}$, $B$ is a standard Brownian motion.

In this part, we will consider the following stochastic differential equation:
\begin{align} \label{sde}
\begin{cases}
   dX_{t}=\displaystyle \sum^{d}_{i=0}V_{i}\left(X_{t}\right)dB^{i}_{t}\\
   X_0=x_0
\end{cases}
\end{align}
where the $V_i$'s are Lipschitz $C^{\infty}$ vector fields on $\mathbb{R}^n$ with Lipschitz derivatives and $B_t$ is a d-dimensional fractional Brownian motion with Hurst parameter $H>1/2$ with the convention that $B_0=t$.

Let us remember that when $H>1/2$, $B_t$ is almost surely H\"{o}lder path with H\"{o}lder exponent strictly less than H. In particular, we can pick $1/2<\beta<H$ such that $B_t$ is $\beta$ H\"{o}lder path. Hence, we can define the Stochastic Taylor expansion associated with (\ref{sde}) just as in the deterministic case by $$\displaystyle x_0+\sum^{\infty}_{k=1}g_{k}(t),$$ where $\displaystyle g^{j}_{k}(t)=\sum_{|I|=k}P^{j}_{I}\int_{\triangle^{k}[0,t]}dB^{I}, \quad P^{j}_{I}=(V_{i_1}\cdots  V_{i_k}\pi^{j})(x_0)$ and $\int_{\triangle^{k}[0,t]}dB^{I}=\displaystyle
    \int_{0<t_1<t_2\cdots<t_k<t}dB^{i_1}_{t_1}\cdots dB^{i_k}_{t_k}$.

\bigskip

As in the previous section, we assume throughout this section assume that the $V_i$'s are Lipschitz $C^{\infty}$ vector fields on $\mathbb{R}^n$ with Lipschitz derivatives and analytic on the set $\{x:  \|x-x_0\| \leq C\}$ for some $C>0$.

Now let us state the following convergence theorem for solution to the stochastic differential equation (\ref{sde}), which is an immediate consequence of  Theorem \ref{general}.

\begin{theorem}
There exists an almost surely positive stopping time $T$, such that for $0\le t <T$ the series
\[
\| x_0 \| +\sum_{i=1}^\infty \| g_k(t) \|
\]
almost surely converges and almost surely
\[
X_t=x_0+\sum_{k=1}^{+\infty} g_k(t).
\]
\end{theorem}

\subsection{Quantitative deterministic bounds}

In this part, we will derive a convergence criterion for the stochastic Taylor expansion associated with the SDE (\ref{sde}). The next theorem is just an direct application of the deterministic theorem \ref{convergence}.

\begin{theorem}
Let $\alpha$ such that $1-H <\alpha <\frac{1}{2}$. Let us assume that there exist $M>0$ and $0<\gamma<1-2\alpha$ such that for every word $I \in \{ 0,\cdots, d\}^k$
\begin{align}\label{criteria}
\|P_{I}\|\leq \Gamma(\gamma|I|)M^{|I|}.
\end{align}
Then, for every $r >1$, $T_C(r)>0$ almost surely, and there exists a constant $ K_{\alpha,\gamma,M,d} >0 $ depending only on the subscript variables such that for every $0 \le  t <T_C(r)$ and $N \ge 1$,
\begin{align}\label{CB1}
\left\|X_t-\left(x_0+\sum^{N}_{k=1}g_{k}(t)\right) \right\|\leq K_{\alpha,\gamma,M,d}  \| B \|_{\alpha,t, \infty}  \frac{(M(d+1)\Lambda_\alpha(t,B) C_\alpha(t) )^N }{\Gamma((1-2\alpha-\gamma)N)} e^{2 (M(d+1)\Lambda_\alpha(t,y) C_\alpha(t) )^{\frac{1}{1-2\alpha-\gamma}}}
\end{align}
In particular for almost  every $0 \le  t <T_C(r)$, we have almost surely
\begin{align}\label{CB2}
X_t=x_0+\sum^{+\infty}_{k=1}g_{k}(t).
\end{align}
Moreover if $C=+\infty$, that is if the $V_i$'s are analytic on $\mathbb{R}^n$, then (\ref{CB1}) and (\ref{CB2}) hold for every $t \ge 0$.
\end{theorem}

We can observe that the error estimate provided by the previous theorem is in $L^p$ for every $p \ge 1$. Indeed, Lemma 7.4 in Nualart-Rascanu (\cite{Nualart}) says for any $\beta < H$ and $T>0$ there exists a positive random variable $\eta_{\beta,T}$ such that $E(|\eta_{\beta,T}|^{p})<\infty$ for all $p \ge1$ and for all $s,t\in [0,T]$

\[
\|B(t)-B(s)\|\leq \eta_{\beta,T}|t-s|^{\beta} \quad a.s.
\]

We now have,
\begin{align*}
\|B\|_{\alpha,T,\infty}=&\sup_{0\leq t\leq T}\left(\|B_t\|+\int^{t}_{0}\frac{\|B_t-B_s\|}{(t-s)^{1+\alpha}}ds\right)\\
                     \leq &\sup_{0\leq t\leq T}\left(\eta_{\beta,T}t^{\beta}+
                     \int^{t}_{0}\eta_{\beta,T}(t-s)^{\beta-1-\alpha}ds\right)\\
                     \leq &\left(T^{\beta}+\frac{1}{\beta-\alpha}T^{\beta-\alpha}\right) \eta_{\beta,T}
\end{align*}
So, $\mathbb{E}\|B\|^{p}_{\alpha,T,\infty}<\infty$ follows from the fact that  $E(|\eta_{\beta,T}|^{p})<\infty$ for all $p\geq 1$ and for all $s,t\in [0,T]$.\\
As for $\Lambda_{\alpha}(T,B)$, we have the following estimate
\begin{align*}
\Lambda_{\alpha}(T,B)=&\frac{1}{\Gamma(1-\alpha)}\sup_{0\leq s\leq t \leq T}\|D^{1-\alpha}_{t-}B_{t-}(s)\|\\
                   \leq & \frac{1}{\Gamma(1-\alpha)\Gamma(\alpha)}\sup_{0\leq s\leq t \leq T}\left(\frac{\|B_t-B_s\|}{(t-s)^{1-\alpha}}+
                   (1-\alpha)\int^{t}_{s}\frac{\|B_u-B_s\|}{(u-s)^{2-\alpha}}du\right)\\
                   \leq & \frac{1}{\Gamma(1-\alpha)\Gamma(\alpha)}\sup_{0\leq s\leq t \leq T}\left(\eta_{\beta,T}(t-s)^{\beta+\alpha-1}+(1-\alpha)\eta_{\beta,T}\int^{t}_{s}(u-s)^{\beta+\alpha-2}du\right)\\
                   =&\frac{\eta_{\beta,T}}{\Gamma(1-\alpha)\Gamma(\alpha)}\sup_{0\leq s\leq t \leq T}\left((t-s)^{\beta+\alpha-1}+\frac{1-\alpha}{\alpha+\beta-1}(t-s)^{\alpha+\beta-1}\right)\\
                   \leq & \left(\frac{\beta T^{\alpha+\beta-1}}{(\alpha+\beta-1)\Gamma(1-\alpha)\Gamma(\alpha)}\right)\eta_{\beta,T}
\end{align*}
Note that by assumption we know $\alpha+\beta-1>0$ and  $\mathbb{E}|\Lambda_{\alpha}(B)|^{p}<\infty$ follows by the fact that $\eta_{\beta,T}$ has a Gaussian tail from Fernique's theorem.

\begin{remark}
Note that in this argument we didn't include the drift term  when we showed the $L^p$ integrability. However, the argument can be modified trivially if we include drift term, since the fractional derivative of $t$ with order $1-\alpha>0$ is bounded in any closed interval $[0,T]$
\end{remark}

\subsection{Quantitative probabilistic bounds, $L^2$ error estimate}

In this part, we will develop a convergence criterion for the stochastic Taylor expansion which is based on the $L_2$ estimates of the iterated integrals of fractional Brownian motion.  Such estimates may be found in  Proposition 4.8 in \cite{Tindel}. They rely on the computation of $\mathbb{E}\left(\int_{\triangle^{m}[0,t]}dB^{I}\right)$ established in \cite{FB1} and the isometry property of the divergence operator.

Compared with the deterministic approach, the probabilistic method improves the speed   of convergence  in $L^2$ when $H \in \left( \frac{1}{2}, \frac{3}{4} \right)$.

\

For convenience in this section we restrict ourselves to the case where there is no drift, that is we assume $V_0=0$.  By using heavier notations, the arguments  we use may be extended when $V_0\neq 0$.

\

Let us first remind the result by Neuenkirch-Nourdin-Tindel in \cite{Tindel}.
\begin{proposition}
Assume that $H>1/2$ and let $\alpha_{H}=H(2H-1)$ and $K=\sqrt{\frac{2}{\alpha_{H}}}$ we have the following inequality:

\[
 \mathbb{E}\left(\left|\int_{\triangle^{m}[0,t]}dB^{I}\right|^2\right) \leq \frac{K^{2m}}{m!} t^{2H |I|}.
\]
\end{proposition}
It is interesting to observe that this decay in $m$ is faster than the decay obtained by the deterministic method when $H \in \left( \frac{1}{2}, \frac{3}{4} \right)$.
We now prove:

\begin{theorem}
Let us assume that there exists $M>0$ and $0\le \gamma <1/2$ such that $\|P_I\|\leq M^{|I|}(|I|!)^{\gamma}$.  Then for every $r>1$, $T_C(r)>0$ almost surely and for $0\le t <T_C(r)$, we have in $L^2$ and  almost surely
\[
X_t=x_0+\sum^{+\infty}_{k=1}g_{k}(t).
\]
Moreover there exists a constant $C_\gamma >0$ such that,
\[
\mathbb{E} \left( \left\|X_t-\left(x_0+\sum^{N}_{k=1}g_{k}(t)\right) \right\|^2 , t <T_C(r) \right)^{1/2} \leq C_\gamma  \frac{ (dM Kt^{2H})^{N+1}}{((N+1)!)^{1/2-\gamma}} \Phi_\gamma ( dK M t^{2H}),
\]
where
\[
\Phi_\gamma (x)= \sum_{k=0}^{+\infty} \frac{x^k}{(k!)^{1/2-\gamma}}.
\]
Moreover if $C=+\infty$, that is if the $V_i$'s are analytic on $\mathbb{R}^n$, then we may take $T_C(r)=+\infty$ in the above inequality.
\end{theorem} 

\begin{proof}
Let $\rho >0$. We have:
\begin{align*}
 \mathbb{P}\left(\sum_{|I|=m}\left\|P_{I}\int_{\triangle^{m}[0,t]}dB^{I}\right\|\geq \rho^m\right) \leq &\frac{1}{\rho^{2m}}\mathbb{E}\left[\left(\sum_{|I|=m}\left\|P_{I}\int_{\triangle^{m}[0,t]}dB^{I}\right\|\right)^2\right]\\
 \leq&\frac{d^m}{\rho^{2m}}\sum_{|I|=m}\mathbb{E}\left(\left\|P_{I}\int_{\triangle^{m}[0,t]}dB^{I}\right\|^2\right)\\
 \leq&\left(\frac{d}{\rho^2}\right)^{m}\sum_{|I|=m}\|P_I\|^{2}\mathbb{E}\left(\left|\int_{\triangle^{m}[0,t]}dB^{I}\right|^2\right)\\
 \leq &\left(\frac{d}{\rho^2}\right)^{m}\sum_{|I|=m}\|P_I\|^{2}\frac{K^{2m}}{m!}t^{2mH} \\
 \leq&\left(\frac{d^2}{\rho^2}\right)^{m}M^{2m} \frac{K^{2m}}{(m!)^{1-2\gamma}}t^{2mH}
\end{align*}

Therefore: $\displaystyle \sum_{m=1} \mathbb{P}\left(\sum_{|I|=m}\left\|P_{I}\int_{\triangle^{m}[0,t]}dB^{I}\right\|\geq \rho^m\right)<\infty $. 
From Borel Cantelli Lemma, we deduce that for every $ r>1$, $\displaystyle \sum^{\infty}_{m=1}r^{m}\|g_{m}(t)\|$ converges uniformly on any interval $[0,T]$ almost surely. From the results of the previous section, it thus implies that for every $r>1$, $T_C(r)>0$ and that  for every $0 \le t <T_C(r)$, we have  almost surely
\[
X_t=x_0+\sum^{+\infty}_{k=1}g_{k}(t).
\]
Finally, we have
\begin{align*}
\mathbb{E} \left( \left\|X_t-\left(x_0+\sum^{N}_{k=1}g_{k}(t)\right) \right\|^2, t <T_C(r) \right)^{1/2} & \le \sum_{k=N+1}^{+\infty} \sum_{|I|=k} \frac{K^k}{\sqrt{k!}} t^{2Hk} M^k (k!)^\gamma \\
 & \le   \sum_{k=N+1}^{+\infty} \frac{(dKt^{2H}M)^k}{(k!)^{1/2-\gamma }} \\
 & \le C_\gamma  \frac{ (dM Kt^{2H})^{N+1}}{((N+1)!)^{1/2-\gamma}} \Phi_\gamma ( dK M t^{2H}).
\end{align*}
\end{proof}

\end{document}